\documentclass{amsart}

\usepackage{amsmath}
\usepackage{amscd}
\usepackage{amssymb}
 \usepackage{color}

\input xy
\xyoption{all}

\newcommand{\pd}{\partial}

\newcommand{\bP}{{\mathbb P}}
\newcommand{\bQ}{{\mathbb Q}}
\newcommand{\bR}{{\mathbb R}}

\newcommand{\bZ}{{\mathbb Z}}

\newcommand{\half}{\frac{1}{2}}

   \newcommand{\bw}{{\mathbf w}}
\newcommand{\bep}{{\mathbf e}}

\newtheorem{theorem/definition}{Theorem/Definition}[section]

\newtheorem{proposition}{Proposition}[section]

\newtheorem{Conjecture}{Conjecture}

\theoremstyle{remark}

\theoremstyle{definition}

\newcommand{\bml}{\begin{multline}}
\newcommand{\eml}{\end{multline}}

\newcommand{\bag}{\begin{align}}
\newcommand{\egn}{ \end{align}}

\newcommand{\be}{\begin{equation}}

\newcommand{\ee}{\end{equation}}
\newcommand{\bea}{\begin{eqnarray}}
\newcommand{\eea}{\end{eqnarray}}
\newcommand{\ben}{\begin{eqnarray*}}
\newcommand{\een}{\end{eqnarray*}}
\newcommand{\bet}{\begin{equation}
\begin{split}}
\newcommand{\eet}{\end{split}
\end{equation}}

\definecolor{yellow}{rgb}{1,1,0}
\definecolor{orange}{rgb}{1,.7,0}
\definecolor{red}{rgb}{1,0,0} \definecolor{blue}{rgb}{0,0,1}
\definecolor{white}{rgb}{1,1,1}

\definecolor{A}{rgb}{.75,1,.75}

\begin{document}

\title
{Integrality Properties of Variations of Mahler Measures}
\author{Jian Zhou}
\address{Department of Mathematical Sciences\\Tsinghua University\\Beijing, 100084, China}
\email{jzhou@math.tsinghua.edu.cn}

\begin{abstract}
We propose some conjectures on the integrality properties related to
the  variation of Mahler measures,
inspired by the results in the elliptic curve case by Rodriguez Villegas, Stienstra and Zagier.
\end{abstract}


\maketitle

In the study of mirror symmetry,
there are some amazing integrality results,
including the integrality of mirror maps (Lian-Yau integrality) \cite{Lia-Yau1, Lia-Yau2, Zud, KR1, KR2, KR3, Del, Zho}
and the integrality of instanton numbers (including Gopakumar-Vafa integrality for closed strings
and Ooguri-Vafa integrality for open strings in arbitrary genera),
see e.g. \cite{KSV, GV, OV, Pen, Kon}.
In this note we will propose some conjectures on the integrality properties related to
the  variation of Mahler measures,
inspired by the results in the elliptic curve case in \cite{RV, Sti, Zag}.
More precisely,
we will identify a quantity $Q(z)$ associated with the variation of Mahler measures
with the local mirror map,
and make some conjectures about the integrality properties of its expression
in term of the mirror parameter $q(z)$
and vice versa.
Some examples are presented.

\section{variations of Mahler Measures, Periods,  Picard-Fuchs Equations and Mirror Maps}

\subsection{One-parameter deformations of Fermat type hypersurfaces in weighted projective spaces}

The geometric objects we will study are deformations of Fermat type Calabi-Yau hypersurfaces of the form:
\be
x_1^{k_1} + \cdots + x_n^{k_n} - k\psi x_1 \cdots x_n = 0
\ee
in a weighted projective space $\bP^{n-1}_{w_1, \dots, w_n}$.
Here $k_1 \leq  \dots \leq k_n$ are positive integers such that
\be \label{eqn:Sum=1}
\frac{1}{k_1} + \cdots + \frac{1}{k_n} = 1,
\ee
$k$ is the least common multiplier of $k_1, \dots, k_n$,
and $w_1 = k/k_1, \dots, w_n = k/k_n$.

For each $n$,
there are only finitely many solutions to \eqref{eqn:Sum=1}.
They can be found by the following search algorithm:
First $k_1$ is bounded between $2$ and $n$,
we search in a reversed order for $k_1$ in this range;
for fixed $k_1$, $k_2$ has the following bound:
$$k_1 \leq k_2 \leq \frac{n-1}{1-\frac{1}{k_1}},$$
we search for $k_2$ in reversed order in this range;
for fixed $k_1, k_2$,
$k_3$ has the following bound:
$$k_2 \leq k_3 \leq \frac{n-2}{1-\frac{1}{k_1}-\frac{1}{k_2}},$$
we search for $k_3$ in reversed order in this range,
and so on.
This algorithm can be easily implemented by a computer algebra system \footnote{The author thanks Dr. Fei Yang
for providing us the Maple codes that implements the search algorithm.}.
The following are the results for $n =2,3,4,5$.
For $n=2$, there is only one solution: $(2,2)$;
For $n=3$, there are $3$ solutions: $(3,3,3)$, $(2,4,4)$, $(2,3,6)$.
For $n=4$, there are $13$ solutions: $(4,4,4,4)$, $(3,4,4,6)$, $(3,3,4,12)$,
$(2,6,6,6)$, $(2,5,5,10)$,
$(2,4,8,8)$, $(2,4,6,12)$, $(2,4,5,20)$,
$(2,3,12,12)$, $(2,3,10,15)$, $(2,3,9,18)$,
$(2,3,8,24)$, $(2,3,7,42)$.
For $n=5$, there are $147$ solutions:

\begin{tabular}{cccc}
$(5,5,5,5,5)$, & $(4,4,6,6,6)$, & $(4,4,5,5,10)$, & $(4,4,4,8,8)$, \\
$(4,4,4,6,12)$,& $(4,4,4,5,20)$, & $(3,6,6,6,6)$, & $(3,5,5,6,10)$, \\
$(3,5,5,5,15)$,& $(3,4,6,8,8)$, & $(3,4,6,6,12)$, & $(3,4,5,6,20)$, \\
$(3,4,5,5,60)$,& $(3,4,4,12,12)$,& $(3,4,4,10,15)$,& $(3,4,4,9,18)$, \\
$(3,4,4,8,24)$, & $(3,4,4,7,42)$, &$(3,3,9,9,9)$, & $(3,3,8,8,12)$, \\
$(3,3,7,7,21)$, &$(3,3,6,12,12)$,& $(3,3,6,10,15)$,& $(3,3,6,9,18)$, \\
$(3,3,6,8,24)$, &$(3,3,6,7,42)$,&$(3,3,5,15,15)$,& $(3,3,5,12,20)$, \\
$(3,3,5,10,30)$, &$(3,3,5,9,45)$,& $(3,3,5,8,120)$,& $(3, 3, 4, 24, 24)$, \\
$(3, 3, 4, 21, 28)$, & $(3,3,4,20,30)$, & $(3,3,4,18,36)$, & $(3,3,4,16,48)$, \\
$(3,3,4,15,60)$, & $(3,3,4,14,84)$, & $(3, 3, 4, 13, 156)$, & $(2,8,8,8,8)$, \\
$(2,7,7,7,14)$,& $(2,6,9,9,9)$, & $(2,6,8,8,12)$, & $(2,6,7,7,21)$,\\
$(2,6,6,12,12)$,& $(2,6,6,10,15)$, &$(2,6,6,9,18)$,& $(2,6,6,8,24)$, \\
$(2,6,6,7,42)$, & $(2,5,10,10,10)$, &$(2,5,8,8,20)$,& $(2,5,7,7,70)$, \\
$(2,5,6,15,15)$, &$(2,5,6,12,20)$, &$(2,5,6,10,30)$, & $(2,5,6,9,45)$, \\
$(2,5,6,8,120)$, &$(2,5,5,20,20)$,& $(2,5,5,15,30)$, & $(2,5,5,14,35)$, \\
$(2,5,5,12,60)$, & $(2, 5, 5, 11, 110)$, & $(2,4,12,12,12)$,& $(2,4,10,12,15)$,\\
$(2,4,10,10,20)$, & $(2,4,9,12,18)$, & $(2,4,9,9,36)$, &$(2,4,8,16,16)$, \\
$(2,4,8,12,24)$, & $(2,4,8,10,40)$, & $(2,4,8,9,72)$, &$(2,4,7,14,28)$, \\
$(2,4,7,12,42)$, & $(2,4,7,10,140)$, & $(2,4,6,24,24)$,& $(2,4,6,21,28)$, \\
$(2,4,6,20,30)$, & $(2,4,6,18,36)$, & $(2,4,6,16,48)$,& $(2,4,6,15,60)$, \\
$(2,4,6,14,84)$, & $(2,4,5,13,156)$, & $(2,4,5,40,40)$,& $(2,4,5,36,45)$, \\
$(2,4,5,30,60)$, & $(2,4,5,28,70)$, & $(2,4,5,25,100)$,& $(2,4,5,24,120)$, \\
$(2,4,5,22,220)$,& $(2,4,5,21,420)$, & $(2,3,18,18,18)$, &$(2,3,16,16,24)$, \\
$(2,3,15,20,20)$, & $(2,3,15,15,30)$, & $(2,3,14,21,21)$, &$(2,3,14,15,35)$, \\
$(2,3,14,14,42)$, &$(2,3,13,13,78)$, & $(2,3,12,24,24)$, &$(2,3,12,21,28)$, \\
$(2,3,12,20,30)$, & $(2,3,12,18,36)$, & $(2,3,12,16,48)$, & $(2,3,12,15,60)$, \\
$(2,3,12,14,84)$,& $(2,3,12,13,156)$, & $(2,3,11,22,33)$, & $(2,3,11,15,110)$, \\
$(2,3,11,14,231)$,& $(2,3,10,30,30)$,& $(2,3,10,24,40)$, & $(2,3,10,20,60)$, \\
$(2,3,10,18,90)$, &$(2,3,10,16,240)$, &$(2,3,9,36,36)$, & $(2,3,9,30,45)$, \\
$(2,3,9,27,54)$, &  $(2,3,9,24,72)$, & $(2,3,9,22,99)$, & $(2,3,9,21,126)$, \\
$(2,3,9,20,180)$, & $(2,3,9,19,342)$,& $(2,3,8,48,48)$, & $(2,3,8,42,56)$, \\
$(2,3,8,40,60)$, & $(2,3,8,36,72)$,  & $(2,3,8,33,88)$, & $(2,3,8,32,96)$,  \\
$(2,3,8,30,120)$, &  $(2,3,8,28,168)$,& $(2,3,8,27,216)$, & $(2,3,8,26,312)$,  \\
$(2,3,8,25,600)$,& $(2,3,7,84,84)$,& $(2,3,7,78,91)$, & $(2,3,7,70,105)$, \\
$(2,3,7,63,126)$,& $(2,3,7,60,140)$,& $(2,3,7,56,168)$, & $(2,3,7,54,189)$, \\
$(2,3,7,51,238)$,&$(2,3,7,49,294)$, &$(2,3,7,48,336)$, & $(2,3,7,46,483)$, \\
 $(2,3,7,45,630)$, &$(2,3,7,44,924)$, &$(2,3,7,43,1806)$.
\end{tabular}

\noindent When $n=6$,
there are $3462$ solutions,
e.g. $(2,7,43,1807,3263442)$.

There are two ways to count the number of solutions to \eqref{eqn:Sum=1} for each $n$.
The first is a simple count,
i.e.,
each solution is counted as $1$.
The second is a weighted count,
i.e., each solution is counted as $1$ over the order of its automorphism group.
By an automorphism of a solution $(k_1, \dots, k_n)$,
we mean a permutation $\sigma \in S_n$ such that
$k_{\sigma(i)} = k_i$ for all $i=1, \dots, n$.
It is interesting to study these counting problems.

\subsection{Variations of Mahler measures}

Given a solution $(k_1, \dots, k_n)$ to \eqref{eqn:Sum=1},
let $k$ be the least common multiplier of $k_1, \dots, k_n$.
Consider a weighted homogeneous polynomial of the form
$ k \psi \prod_{i=1}^n x_i - P(x_1, \dots, x_n)$,
where
$$
P(x_1, \dots, x_n):= \sum_{i=1}^n x_i^{k_i}
$$
with $\psi$ a complex parameter.
This is a weighted homogeneous polynomial of degree
$$k_1 w_1 = \cdots = k_n w_n = w_1 + \cdots + w_n = k,$$
it defines a Calabi-Yau hypersurface $X_\psi$ in the weighted projective space
$\bP^{n-1}_{\bw}$, where $\bw = (w_1, \dots, w_n)$.
For $\bep=(\epsilon_1, \dots, \epsilon_{n-1}) \in \bR_+^{n-1}$,
consider the following $(n-1)$-cycle $C_\bep$ in $\bP^{n-1}_\bw$ defined by:
\be
|x_1|= \epsilon_1, \dots, |x_{n-1}| = \epsilon_{n-1}, x_n = 1.
\ee
Consider the following integral over this cycle:
\be \label{eqn:Q}
\tilde{M} :=
\exp\left(-\frac{1}{(2\pi i)^{n-1}}\oint_{C_{\bep}}
\log (\psi-\frac{P(x_1,\dots, x_{n-1},1)}{k x_1 \cdots x_{n-1}})\,
\frac{dx_1}{x_1} \cdots \frac{dx_{n-1}}{x_{n-1}}\right).
\ee
Recall the \emph{logarithmic Mahler measure} $m(F)$ and the \emph{Mahler measure} $M(F)$ of a Laurent polynomial $F(x_1,\dots, x_{n-1})$ with complex coefficients are:
\begin{eqnarray}
m (F)&:=&\frac{1}{(2\pi i)^{n-1}}\oint\cdots \oint_{|x_1|=\epsilon_1, \dots, |x_{n-1}|=\epsilon_{n-1}}
\log |F|\,
\prod_{i=1}^{n-1} \frac{dx_i}{x_i},\\
M(F)&:=&\exp (m(F))\;.
\end{eqnarray}
One then finds
\be
M(F_\psi)=|\tilde{M}|^{-1},
\ee
where $F_\psi(x_1, \dots, x_{n-1}) = \psi-\frac{P(x_1,\dots, x_{n-1},1)}{k x_1 \cdots x_{n-1}}$.
In the case of elliptic curves \cite{RV},
the Mahler measure is related to the special values of the L-function associated to $X_\psi$
by Beilinson Conjectures.
Similar relationship is expected in higher dimensions.

By taking expansion in $\xi = \psi^{-1}$,
one gets from \eqref{eqn:Q}:
\ben
\tilde{M} &=& \xi \exp \biggl( \sum_{m=1}^\infty
\frac{\xi^m}{mk^m} \frac{1}{(2\pi i)^2} \oint_{C_\bep}
\frac{(\sum_{i=1}^{n-1} x_i^{k_i}+1)^m}{(x_1 \cdots x_{n-1})^m}\, \prod_{j=1}^{n-1} \frac{dx_j}{x_j}
\biggr).
\een
Thus
\be
\tilde{M} =\xi \exp\left(\sum_{m=1}^\infty c_m\frac{\xi^m}{m k^m}\right)
\ee
with $c_m$ the coefficient of  $x_1^m \cdots x_{n-1}^m$ in $(\sum_{i=1}^{n-1}x_i^{k_i}+1)^m$.
In particular,
$Q$ is independent of the choices of $\epsilon_1, \dots, \epsilon_{n-1}$.
By multinomial formula,
one easily gets:
\be
c_m = \begin{cases}
\frac{m!}{\prod_{i=1}^n (m/k_i)!}, & k|m, \\
0, & \text{otherwise}.
\end{cases}
\ee
Let $Q = \tilde{M}^k/k^k$, and $z = \xi^k/k^k$,
then we have
\be \label{def:Q}
Q = z \exp \biggl( \sum_{m=1}^\infty \frac{(km)!}{\prod_{i=1}^n (w_im)!} \frac{z^m}{m} \biggr).
\ee

\subsection{Periods and Picard-Fuchs equations}

Let $\theta = z\frac{d}{d z}$.
Differentiating \eqref{eqn:Q} and \eqref{def:Q} one finds
\bea
\theta \log Q
&=&\frac{\psi}{(2\pi i)^{n-1}} \oint_{C_\epsilon}
\frac{dx_1 \cdots dx_{n-1}}{k\psi x_1 \cdots x_{n-1}- P(x_1, \dots, x_{n-1},1)}\\
& = & \sum_{m=0}^\infty \frac{(km)!}{\prod_{i=1}^n (w_im)!} z^m.
\eea
Thus $\theta \log Q$ is a period of a family $\omega_\psi$ of holomorphic forms on $X_\psi$.

Write
\be
\alpha_m := \frac{(km)!}{\prod_{i=1}^n (w_im)!}
= \frac{k^{km} \prod_{j=0}^{km-1} (m - \frac{j}{k})}{\prod_{i=1}^n [w_i^{w_i} \prod_{j=0}^{w_im-1} (m - \frac{j}{w_i})]},
\ee
Then we have
\be
\frac{\alpha_{m}}{\alpha_{m-1}}
= \frac{k^k \prod_{j=0}^{k-1} (m - \frac{j}{k})}{\prod_{i=1}^n [w_i^{w_i} \prod_{j=0}^{w_i-1} (m - \frac{j}{w_i})]}
= \frac{k^k}{\prod_{i=1}^n w_i^{w_i}} \prod_{j=1}^l \frac{m - 1 + a_j}{m -b_j},
\ee
where in the second equality we remove the common factors of the numerator and the denominator.
The equality
\be \label{eqn:AB1}
\frac{k^k}{\prod_{i=1}^n w_i^{w_i}} \prod_{j=1}^l \frac{m - 1 + a_j}{m -b_j}
= \frac{(km)!}{(k(m-1))!} \cdot \prod_{i=1}^n \frac{(w_i(m-1))!}{(w_im)!}
\ee
and
\be \label{eqn:AB2}
\sum_{j=1}^l (\frac{1}{m-1+a_j} - \frac{1}{m-b_j})
= \sum_{j=0}^{k-1} \frac{1}{m-\frac{j}{k}} - \sum_{i=1}^n \sum_{j=0}^{w_i-1} \frac{1}{m - \frac{j}{w_i}}
\ee
will be of use below.

The recursion relation
\be
\prod_{j=1}^l (m - b_j) \alpha_m = \frac{k^k}{\prod_{i=1}^n w_i^{w_i}} \cdot \prod_{j=1}^l (m-1+a_j) \alpha_{m-1}
\ee
is equivalent to the following Picard-Fuchs equation:
\be \label{eqn:PF1}
\prod_{j=1}^l (\theta - b_j) \Phi = z \frac{k^k}{\prod_{i=1}^n w_i^{w_i}} \cdot \prod_{j=1}^l (\theta+a_j) \Phi
\ee
satisfied by $\theta \log Q$.

The recursion relation
\be
\frac{\alpha_{m}}{\alpha_{m-1}}
= \frac{k^k \prod_{j=0}^{k-1} (m - \frac{j}{k})}{\prod_{i=1}^n [w_i^{w_i} \prod_{j=0}^{w_i-1} (m - \frac{j}{w_i})]}
\ee
can also be rewritten as
\be
\prod_{i=1}^n \prod_{j=0}^{w_i-1} (m - \frac{j}{w_i}) \cdot \alpha_m
= \frac{k^k}{\prod_{i=1}^n w_i^{w_i}} \cdot  \prod_{j=0}^{k-1} (m - 1 + 1 - \frac{j}{k}) \alpha_{m-1}.
\ee
It is equivalent to the Picard-Fuchs equation:
\be \label{eqn:PF3}
\prod_{i=1}^n \prod_{j=0}^{w_i-1} (\theta - \frac{j}{w_i}) \cdot \Phi
= z \frac{k^k}{\prod_{i=1}^n w_i^{w_i}} \cdot  \prod_{j=0}^{k-1} (\theta  + 1 - \frac{j}{k}) \Phi.
\ee

In some cases one has
\be
\frac{\alpha_m}{\alpha_{m-1}}
= \frac{k^{k}}{\prod_{i=1}^n w_i^{w_i}}  \frac{\prod_{j=1}^{n-1} (m-1+a_j)}{m^{n-1}},
\ee
where $\{a_1, \dots, a_{n-1}\}$ is obtained from the set
$\{ \frac{1}{k}, \frac{2}{k}, \dots, \frac{k-1}{k}\}$
by removing integral multiples of $\frac{1}{k_i} = \frac{w_i}{k}$, where $w_i > 1$.
In this case
the Picard-Fuchs equation takes the following form:
\be \label{eqn:PF2}
\theta^{n-1} \Phi = \frac{k^kz}{\prod_{i=1}^n w_i^{w_i}} \prod_{j=1}^{n-1} (\theta + a_j) \Phi.
\ee
This happens if and only if $(w_i, w_j) = 1$ for $i \neq j$.

For $n=2$,
the only case $(k_1, k_2) =(2, 2)$ has this property.
The Picard-Fuchs equation is
\be
\theta \Phi - z (\theta+\half) \Phi = 0.
\ee
For $n=3$,
all cases of solutions to \eqref{eqn:Sum=1} has this property.
The Picard-Fuchs operators are:
\bea
&& \theta^2 - z (\theta+\frac{1}{3})(\theta+\frac{2}{3}), \qquad (k_1, k_2, k_3) = (3,3,3), \\
&& \theta^2 - z (\theta+\frac{1}{4})(\theta+\frac{3}{4}), \qquad (k_1, k_2, k_3) = (2,2,4), \\
&& \theta^2 - z (\theta+\frac{1}{6})(\theta+\frac{5}{6}), \qquad (k_1, k_2, k_3) = (2,3,6).
\eea
For $n=4$, we have the following cases:
\bea
&& \theta^3 - z (\theta+\frac{1}{4})(\theta+\frac{2}{4})(\theta+\frac{3}{4}), \qquad (k_1, k_2, k_3,k_4) = (4,4,4,4), \\
&& \theta^3 - z (\theta+\frac{1}{6})(\theta+\frac{3}{6})(\theta+\frac{5}{6}), \qquad (k_1, k_2, k_3,k_4) = (2,6,6,6).
\eea
For $n=5$ we have the following cases:
\bea
&& \theta^4 - z (\theta+\frac{1}{5})(\theta+\frac{2}{5})(\theta+\frac{3}{5})(\theta+\frac{4}{5}), \qquad \vec{k} = (5,5,5,5,5), \\
&& \theta^4 - z (\theta+\frac{1}{6})(\theta+\frac{2}{6})(\theta+\frac{4}{4})(\theta+\frac{5}{6}), \qquad \vec{k} = (3,6,6,6,6), \\
&& \theta^4 - z (\theta+\frac{1}{8})(\theta+\frac{3}{8})(\theta+\frac{5}{8})(\theta+\frac{7}{8}), \qquad \vec{k} = (2,8,8,8,8), \\
&& \theta^4 - z (\theta+\frac{1}{10})(\theta+\frac{3}{10})(\theta+\frac{7}{10})(\theta+\frac{9}{10}), \qquad \vec{k} = (2,5,10,10,10).
\eea

\subsection{Logarithmic solutions and mirror maps}

Equation \eqref{eqn:PF1} has a solution of logarithmic behavior:
\be
g_1(z) = g_0(z) \cdot \log z + h(z),
\ee
where $g_0(z) = \theta \log Q = \sum_{m \geq 0} \frac{(km)!}{\prod_{i=1}^n (w_im)!} z^m$ and
$h(z) = \sum_{m \geq 1} \gamma_m z^m$.
Rewrite \eqref{eqn:PF2} as
\bea
&& \prod_{j=1}^l (\theta-b_j) h(z) = \frac{k^k}{\prod_{i=1}^n w_i^{w_i}} z \prod_{j=1}^l (\theta + a_j) h(z) \\
&& +  \frac{k^k}{\prod_{i=1}^n w_i^{w_i}} z \sum_{i=1}^l \frac{\prod_{j=1}^l (\theta+a_j)}{\theta + a_i}  g_0(z)
 - \sum_{i=1}^l \frac{\prod_{j=1}^l (\theta-b_j)}{\theta-b_i} g_0(z). \nonumber
\eea
This is equivalent to the following initial value
\be
\gamma_1  = \sum_{i=1}^n (\frac{1}{a_i} - \frac{1}{1-b_i}) \frac{k!}{\prod_{i=1}^n w_i!}
\ee
and recursion relation:
\bea
&& \prod_{j=1}^l (m-b_j) \cdot \gamma_m
= \frac{k^k}{\prod_{i=1}^n w_i^{w_i}}  \prod_{j=1}^l (m -1 + a_j) \cdot \gamma_{m-1} \\
&& - \sum_{i=1}^l \frac{\prod_{j=1}^l (m-b_j)}{m-b_i} \frac{(km)!}{\prod_{i=1}^n (w_im)!} \nonumber \\
&& +  \frac{k^k}{\prod_{i=1}^n w_i^{w_i}}  \sum_{i=1}^l \frac{\prod_{j=1}^l (m-1+a_j)}{m-1-a_i} \frac{k(m-1))!}{\prod_{i=1}^n (w_i(m-1))!}. \nonumber
\eea
Dividing both sides by $\prod_{j=1}^l (m-b_j)$ and making use of \eqref{eqn:AB1} and \eqref{eqn:AB2},
one gets
\bea
&& \gamma_m = \frac{(km)!}{(k(m-1))!} \cdot \prod_{i=1}^n \frac{(w_i(m-1))!}{(w_im)!} \cdot \gamma_{m-1} \\
&& + \sum_{i=1}^l (\frac{1}{m-1+a_i} - \frac{1}{m-b_i}) \frac{(km)!}{\prod_{i=1}^n (w_im)!}. \nonumber
\eea
The solution is given by
\bea
\gamma_m
& = & \sum_{j=1}^m \sum_{i=1}^l (\frac{1}{j-1+a_i} - \frac{1}{j-b_i}) \cdot \frac{(km)!}{\prod_{i=1}^n (w_im)!} \\
& = & \sum_{j=1}^m (\sum_{a=0}^{k-1} \frac{1}{j-\frac{a}{k}}
- \sum_{i=1}^n \sum_{a=0}^{w_i-1} \frac{1}{j - \frac{a}{w_i}}) \cdot \frac{(km)!}{\prod_{i=1}^n (w_im)!}.
\eea
One can also derive this solution from \eqref{eqn:PF3}.

The \emph{mirror map} is defined by
\be
q\: :=\:\exp\left(\frac{g_1(z)}{g_0(z)}\right) = z \exp (h(z)/g_0(z)).
\ee

\subsection{A related Picard-Fuchs system and its mirror map}
In this section we will relate $Q$ to the mirror map of the following Picard-Fuchs equation related to
\eqref{eqn:PF3}:
\be \label{eqn:PF4}
\prod_{i=1}^n \prod_{j=0}^{w_i-1} (\theta - \frac{j}{w_i}) \cdot \Phi
= z \frac{k^k}{\prod_{i=1}^n w_i^{w_i}} \cdot  \prod_{j=0}^{k-1} (\theta  + \frac{j}{k}) \Phi.
\ee
Clear $\Phi = 1$ is a solution,
and we have  the following logarithmic solution:
\be
\Phi_1 = \log z + \sum_{m=1}^\infty \frac{(km)!}{\prod_{i=1}^n (w_im)!} \frac{z^m}{m}.
\ee
The corresponding mirror map is defined by
$$Q = e^{\Phi_1}.$$
Note this is exactly the map $Q$ defined in \eqref{def:Q}.

For example,
when $(w_1,w_2,w_3)=(1,1,1)$,
this is the Picard-Fuch system associated with the local $\bP^2$ geometry \cite{CKYZ},
i.e. the canonical line bundle $\kappa_{\bP^2}$.
In general,
the Picard-Fuchs system \eqref{eqn:PF4} is associated with the local Calabi-Yau geometry
of $\kappa_{\bP^{n-1}_{w_1, \dots, w_n}}$.
Hence we will refer to the mirror map $Q$ as the local mirror map.

\section{Integrality Properties of Variation of Mahler Measures}

It is expected that $z$, $g_0(z)$, $\frac{d}{dq} \log Q$ are modular forms
for the monodromy group of the Picard-Fuchs equation,
and often they can be expressed in terms of usual modular forms.
See \cite{RV, Sti, Zag} for examples in the elliptic curve case.
Our conjectures below are inspired by the results in these papers.
We focus on the integrality properties in this paper
and leave the modular properties to future investigations.

We have
$q = z e^{h(z)/g_0(z)}$ and $Q = z e^{f_n(z)}$,
where
$$f_n(z) =  \sum_{m=1}^\infty
\frac{(mk)!}{\prod_{i=1}^n (w_im)!} \frac{z^m}{m}.$$

\begin{proposition}
One has $q, Q \in z+ z\bZ[[z]]$.
\end{proposition}

\begin{proof}
By a result in \cite{Zho},
we have $(z^{-1}Q)^{1/k} \in 1+ z\bZ[[z]]$.
By the main result in \cite{Del},
to see $q \in z+ z\bZ[[z]]$ one has to show that
\be
[kx] -\sum_{i=1}^n [w_ix] \geq 1
\ee
for $x \in [\frac{1}{k}, 1)$,
where $[x]$ means the integral part of $x$,
i.e.,
$[x]$ is an integer such that $[x] \leq x < [x]+1$,
with equality if and only if $x \in \bZ$.
Therefore,
\be
\sum_{i=1}^n [w_ix] \leq \sum_{i=1}^n w_i x = k x,
\ee
with equality if and only if
$w_ix \in \bZ$ for all $i=1, \dots, n$.
Therefore,
one has
\be
[kx] - \sum_{i=1}^n [w_ix] \geq 0
\ee
for all $x$.
This function is right continuous and jumps at $j/k$, $j=1, \dots, k-1$.
So it suffices to check
\be
j - \sum_{i=1}^n [w_ij/k] > 0
\ee
for all $j=1, \dots, k-1$.
If $\sum_{i=1}^n [w_ij/k] = j$ for some $j=1, \dots, k-1$,
then we have $$w_ij/k =a_i$$ for some integer $a_i$ for all $i=1, \dots, n$.
This means
\be
j = a_i \frac{k}{w_i} = a_i k_i,
\ee
i.e., $j$ is a common multiplier of $k_1, \dots, k_n$,
hence $j \geq k$.
A contradiction.
\end{proof}

\begin{Conjecture}
We have $(z^{-1}q)^{1/k} \in \bZ[[z]]$.
\end{Conjecture}

Using the Lagrange-Good inversion formula \cite{Good} as in  \cite{Zho}
one finds $z= \sum_{m=1}^\infty a_m q^m$ and $z = \sum_{m=1}^\infty A_m Q^m$,
where
\be \label{eqn:zinq}
a_m = \text{the coefficient of $z^{m-1}$ in} \; (1+ \theta (h(z)/g_0(z)) \cdot e^{-m h(z)/g_0(z)},
\ee
and
\be \label{eqn:zinQ}
A_m = \text{the coefficient of $z^{m-1}$ in} \; (1 + \theta f_n(z)) \cdot e^{-m f_n(z)}.
\ee
These coefficients are also integers,
i.e., $z \in q + q \bZ[[q]]$ and $z \in Q + Q \bZ[[Q]]$.
Now we have $Q=z+O(z^2)$ and $q= z + O(z^2)$,
so one can eliminate $z$ and use \eqref{eqn:zinq} and \eqref{eqn:zinQ} to
express $Q$ as a function of $q$ and vice versa.
It is easy to see that $Q \in q \bZ[[q]]$ and $q \in Q \bZ[[Q]]$.
Write
\be
g_0(z) = 1+ \sum_{m=0}^\infty c_m q^m =  1+ \sum_{m=0}^\infty C_m Q^m.
\ee
Then the coefficients $\{c_m\}_{m \geq 1}$ and $\{C_m\}_{m \geq 1}$  are integers.

Note
\be
q \frac{d}{dq} \log Q = z\frac{d}{dz} \log Q \cdot \frac{q}{z}\frac{dz}{dq}
= g_0(z) \cdot \frac{q}{z} \frac{dz}{dq}.
\ee
Because
\be
z \frac{d \log q}{dz}
= 1+ \theta( \frac{h(z)}{g_0(z)})
= 1+ \frac{h(z) \theta g_0(z) - g_0(z) \theta h(z)}{g_0(z)^2}.
\ee
Therefore,
\be
q \frac{d}{dq} \log Q = \frac{g_0(z)}{1+ \theta( \frac{h(z)}{g_0(z)})}
= \frac{g^3_0(z)}{g^2_0(z) + h(z) \theta g_0(z) - g_0(z) \theta h(z)}.
\ee
It follows that $q \frac{d}{dq} \log Q$ lies in $\bQ[[z]]$ hence in $\bQ[[q]]$.
Write
\be
q \frac{d}{dq} \log Q = 1+ \sum_{m=1}^\infty u_m q^m
\ee
and define
\be
b_m = - \frac{1}{m^2} \sum_{d|m} \mu(n/d) u_d
\ee
and
\be
\hat{b}_m = - \frac{1}{m^2} \sum_{d|m} \mu(n/d) (-1)^d u_d.
\ee
Equivalently,
\be
q\frac{d}{d q}\log Q=1-\sum_{m\geq 1} b_m\frac{m^2 q^m}{1- q^m}
=1-\sum_{m\geq 1} \hat{b}_m\frac{m^2 (-q)^m}{1- (-q)^m}.
\ee

\begin{Conjecture}
The numbers $b_m$ and $\hat{b}_m$ are \emph{integers} so that
\be
Q=q\prod_{m\geq 1} (1-q^m)^{mb_m}=q\prod_{m\geq 1} (1-(-q)^m)^{m\hat{b}_m}.
\ee
\end{Conjecture}

Similarly from
\be
Q \frac{d}{dQ} \log q = \frac{Q}{z} \frac{dz}{dQ} \cdot z\frac{d}{d z} \log q,
\ee
and
\be
\frac{z}{Q} \frac{dQ}{dz} = z \frac{d}{dz} \log Q = g_0(z)
\ee
we get:
\be
Q \frac{d}{dQ} \log q = \frac{1+ \theta( \frac{h(z)}{g_0(z)})}{g_0(z)}
= \frac{g^2_0(z) + h(z) \theta g_0(z) - g_0(z) \theta h(z)}{g^3_0(z)}.
\ee
It follows that $Q \frac{d}{dQ} \log q$ lies in $\bQ[[z]]$ hence in $\bQ[[Q]]$.
Write
\be
Q \frac{d}{dQ} \log q = 1+ \sum_{m=1}^\infty v_m Q^m
\ee
and define
\be
c_m = - \frac{1}{m^2} \sum_{d|m} \mu(n/d) v_d
\ee
and
\be
\hat{c}_m = - \frac{1}{m^2} \sum_{d|m} \mu(n/d) (-1)^d v_d.
\ee
Equivalently,
\be
Q\frac{d}{d Q}\log q=1-\sum_{m\geq 1} c_m\frac{m^2 Q^m}{1- Q^m}
=1-\sum_{m\geq 1} \hat{c}_m\frac{m^2 (-Q)^m}{1- (-Q)^m}.
\ee

\begin{Conjecture}
The numbers $c_m$ and $\hat{c}_m$ are \emph{integers} so that
\be
q=Q\prod_{m\geq 1} (1-Q^m)^{mc_m}=Q\prod_{m\geq 1} (1-(-Q)^m)^{m\hat{c}_m}.
\end{equation}
\end{Conjecture}

We have written a Maple algorithm to automate the calculations of the numbers $b_m, \hat{b}_m c_m, \hat{c}_m$
and verify their integrality in various cases.
Some results are presented in the following sections.

\section{Examples}

\subsection{The $n=2$ case}

There is only one possibility:
\be
x_1^2+x_2^2=2\psi x_1x_2.
\ee
Geometrically,
$X_\psi$ is just two points in $\bP^1$.
The Picard-Fuchs operator is given by
\be
L = \theta - 2^2 z (\theta+\frac{1}{2}),
\ee
where $z= (2\psi)^{-2}$, $\theta = z \frac{\pd}{\pd z}$.
It follows that
\bea
&& g_0(z) = \sum_{m=0}^\infty \frac{(2m)!}{(m!)^2} z^m = \frac{1}{\sqrt{1-4z}}, \\
&& g_1(z) = \log z \cdot \sum_{m=0}^\infty \frac{(2m)!}{(m!)^2} z^m
+ \sum_{m=1}^\infty \frac{(2m)!}{(m!)^2}
\cdot \sum_{k=1}^m (\frac{1}{k-1/2} - \frac{1}{k}) \cdot z^m, \\
&& Q(z) = z \exp  \sum_{m=1}^\infty \frac{(2m)!}{(m!)^2} \frac{z^m}{m}
= \frac{4z}{(1+\sqrt{1-4z})^2}.
\eea
From the last equality one easily finds
\be
z= \frac{Q}{(1+Q)^2},
\ee
and so
\be
g_0(z) = \frac{1+Q}{1-Q}.
\ee
Our Maple algorithm indicates that
\be
Q = q.
\ee
I.e.,
\be
\sum_{m=1}^\infty \frac{(2m)!}{(m!)^2} \frac{z^m}{m}
\cdot  \sum_{m=0}^\infty \frac{(2m)!}{(m!)^2} z^m
=  \sum_{m=1}^\infty \frac{(2m)!}{(m!)^2}
\cdot \sum_{k=1}^m (\frac{1}{k-1/2} - \frac{1}{k}) \cdot z^m,
\ee
or equivalently, for $m \geq 1$,
\be
\sum_{a=1}^m \frac{1}{a} \binom{2a}{a} \cdot \binom{2m-2a}{m-a}
= \binom{2m}{m} \sum_{k=1}^m (\frac{1}{k-1/2} - \frac{1}{k}).
\ee
This does not seem to be easy to establish.
Another equivalent formulation is
\be
\sum_{m=1}^\infty \frac{(2m)!}{(m!)^2}
\cdot \sum_{k=1}^m (\frac{1}{k-1/2} - \frac{1}{k}) \cdot z^m
= \frac{1}{\sqrt{1-4z}} \log \frac{4}{(1+\sqrt{1-4z})^2}.
\ee
This does not seem to be easy to establish either.

\subsection{The $n=3$ case}

There are $3$ possibilities,
corresponding to elliptic curves in weighted projective planes.
They have been studied in \cite{RV, Sti, Zag},
which are the source of inspirations of this work.
For
\be
x_1^3+x_2^3+x_3^3=3\psi x_1x_2x_3
\ee
we have
\begin{center}
\begin{tabular}{c|c|c|c|c|c}
$m$ & $b_m$ & $\hat{b}_m$ & $c_m$ & $\hat{c}_m$ & $\hat{c}_m/m$   \\
\hline
1 &  9 &  -9 & -9 & 9 & 9\\
2 & -9 &  -9/2 & -63/2 & -36 & -18 \\
3 & 0  & 0 & -243 & 243 & 81 \\
4 & 9  & 9 & -2304 & -2304 & -576 \\
5 & -9 & 9 & -25425 & 25425 & 5085 \\
6 & 0  & 0 & -614061/2 & -307152 & -51192 \\
7 & 9  & -9 & -3957534 & 3957534 & 565362 \\
8 & -9 & -9 & -53475840 & -5347840 & -6684480 \\
9 & 0  & 0 & -749220273 & 749220273 & 83246697 \\
10 & 9 & 9/2 & -21600703575/2 & -10800364500 & -1080036450
\end{tabular}
\end{center}

For the elliptic curve
\be
x_1^2+x_2^4+x_3^4=4\psi x_1x_2x_3
\ee
\begin{center}
\begin{tabular}{c|c|c}
$m$ & $b_m$ & $\hat{b}_m$  \\
\hline
1 &  28 & -28  \\
2 & -134 & -120  \\
3 & 996 & -996  \\
4 & -10720 & -10720  \\
5 & 139292 & -139292  \\
6 & -2019450 & -2018952 \\
7 & 31545316 & -31545316 \\
8 & -520076672 & -520076672  \\
9 & 8930941980 & -8930941980 \\
10 & -158342776966 & -158342707320
\end{tabular}
\end{center}

\begin{center}
\begin{tabular}{c|c|c|c|c}
$m$ & $c_m$ & $c_m/m$ & $\hat{c}_m$ & $\hat{c}_m/m$   \\
\hline
1 & -28 & -28 & 28 &  28 \\
2 & -258 & -129 & -272 & -136 \\
3 & -4860 & -1620 & 4860 & 1620 \\
4 & -116864 & -29216 & -116864 & -29216 \\
5 & -3259600 & -651920 & 3259600 & 651920 \\
6 & -99763218 & -16627203 &  -99765648 & -16627608 \\
7 & -3256509228 & -465215604 & 3256509228 & 465215604 \\
8 & -111422514176 & -13927814272 & -111422514176 & -13927814272 \\
9 & -3951764383896 & -439084931544 & 3951764383896 & 439084931544 \\
10 & -144178140979800 & -14417814097980 & -144178142609600 & -14417814260960
\end{tabular}
\end{center}

For the elliptic curve
\be
x_1^2+x_2^3+x_3^6=6\psi x_1x_2x_3
\ee
\begin{center}
\begin{tabular}{c|c|c}
$m$ & $b_m$ & $\hat{b}_m$  \\
\hline
1 & 252 & -252  \\
2 & -13374 & -13248  \\
3 & 1253124 & -1253124 \\
4 & -151978752 & -151978752 \\
5 & 21255487740 & -21255487740  \\
6 & -3255937602498 & -3255936975936 \\
7 & 531216722607876 & -531216722607876 \\
8 & -90773367805541376 & -90773367805541376  \\
9 & 16069733941012586748 & -16069733941012586748  \\
 10 & -2925411405456230806590 & -2925411405445603062720
\end{tabular}
\end{center}

\begin{center}
\begin{tabular}{c|c|c}
$m$ & $b_m/m$ & $\hat{b}_m/m$ \\
\hline
1 & 252 & -252 \\
2 & -6687 & -6624 \\
3 & 417708 & -417708 \\
4 & -37994688 & -37994688 \\
5 & 4251097548 & -4251097548 \\
6 & -542656267083 & -542656162656 \\
7 & 531216722607876/7 & -531216722607876/7\\
8 & -11346670975692672 & -11346670975692672 \\
9 &  1785525993445842972 & -1785525993445842972 \\
10 & -292541140545623080659 & -292541140544560306272
\end{tabular}
\end{center}

\begin{center}
\begin{tabular}{c|c|c}
$m$ & $c_m$ & $\hat{c}_m$  \\
\hline
1 & -252 & 252 \\
2 & -18378 & -18504  \\
3 &  -2545884 & 2545884 \\
4 & -457060032 & -457060032 \\
5 & -94790322000 & 94790322000 \\
6 & -21537521398170 & -21537522671112 \\
7 & -5211710079116940 & 5211710079116940 \\
8 & -1320613559984014848 & -1320613559984014848  \\
9 & -346614112277503632216 & 346614112277503632216 \\
10 & -93531635843711988483000 & -93531635843759383644000
\end{tabular}
\end{center}

\begin{center}
\begin{tabular}{c|c|c}
$m$ & $c_m/m$ & $\hat{c}_m/m$   \\
\hline
1 & -252 & 252 \\
2 & -9189 & -9252 \\
3 & -848628 & 848628 \\
4 & -114265008 & -114265008 \\
5 & -18958064400 & 18958064400 \\
6 & -3589586899695 & -3589587111852 \\
7 & -744530011302420 & 744530011302420 \\
8 & -165076694998001856 & -165076694998001856 \\
9 & -38512679141944848024 & 38512679141944848024 \\
10 &-9353163584371198848300 & -9353163584375938364400
\end{tabular}
\end{center}

\subsection{The $n=4$ case}
For the K3 surface
\be
x_1^4+\cdots +x_4^4=4\psi x_1 \cdots x_4
\ee
we have

\begin{center}
\begin{tabular}{c|c|c|c|c}
$m$ & $b_m$ & $\hat{b}_m$ & $b_m/m$ & $\hat{b}_m/m$ \\
\hline
1 & 80 & -80 & 80 & -80 \\
2 & 80 & 120 & 40 & 60 \\
3 & 240 & -240 & 80 & -80 \\
4 & 160 & 160 & 40 & 40 \\
5 & 400 & -400 & 80 & -80 \\
6 & 240 & 360 & 40 & 60 \\
7 & 560 & -560 & 80 & -80 \\
8 & 320 & 320 & 40 & 40 \\
9 & 720 & - 720 & 80 & -80 \\
10 & 400 & 600 & 40 & 60
\end{tabular}
\end{center}

\begin{center}
\begin{tabular}{c|c|c}
$m$ & $c_m$ & $\hat{c}_m$  \\
\hline
1 & -80 & 80 \\
2 & -3280 & -3320 \\
3 & -272240 &  272240 \\
4 & -29945760 & -29945760 \\
5 & -3860155600 & 3860155600 \\
6 & -550279367920 & -550279504040 \\
7 & -84101456589360 & 84101456589360 \\
8 & -13526805760545600 & -13526805760545600 \\
9 & -2262255520889560560 & 2262255520889560560 \\
10 & -390188833066192395600 & -390188833068122473400
\end{tabular}
\end{center}

\begin{center}
\begin{tabular}{c|c|c}
$m$ & $c_m/m$ & $\hat{c}_m/m$   \\
\hline
1 & -80 & 80 \\
2 & -1640 & -1660 \\
3 &-272240/3 & 272240/3 \\
4 & -7486440  & -7486440 \\
5 & -772031120 & 772031120 \\
6 & -275139683960/3 &  -275139752020/3 \\
7 & -12014493798480 & 12014493798480 \\
8 & -1690850720068200 & -1690850720068200 \\
9 & -754085173629853520/3 & 754085173629853520/3 \\
10 &-39018883306619239560 & -39018883306812247340
\end{tabular}
\end{center}

We have also verify the case of
\be
x_1^4+x_2^3+x_2^3+x_4^2-12\psi x_1 \cdots x_4 = 0.
\ee
It turns out that $b_m/m$, $\hat{b}_m/m$, $c_m/m$ and $\hat{c}_m/m$ are all integers.
The numbers are too large to reproduce here.
For example,
$$b_5= 31088578606413096899258654040.$$

\subsection{The $n=5$ case}

For the case of
\be
x_1^5+\cdots +x_5^5=5\psi x_1 \cdots x_5
\ee
we have checked that  $b_m$, $\hat{b}_m$, $c_m$ and $\hat{c}_m$
are all integers divisible by $5$, e.g.,
$$b_5=25050301099750,$$
but not all $b_m/m$, $\hat{b}_m/m$, $c_m/m$ and $\hat{c}_m/m$ are integers.
For example,
$$b_7/7=31249534645239703150/7.$$

We have also checked the case of
\be
x_1^3+x_2^3+x_3^2+x_4^2+x_5^2=12\psi x_1 \cdots x_5
\ee
The numbers $b_m/m$, $\hat{b}_m/m$, $c_m/m$ and $\hat{c}_m/m$ are all integers.
For example,
$$\frac{b_6}{6}=-61961714940992690898780121741257228991904436.$$

\subsection{The $n>5$ cases}

We have also checked various $n> 5$ cases,
e.g.
the case of
\be
x_1^6+\cdots +x_6^6=6\psi x_1 \cdots x_6
\ee
and the case of
\be
x_1^7+\cdots +x_7^7=7\psi x_1 \cdots x_7.
\ee
We conjecture that all $b_m/m$, $\hat{b}_m/m$, $c_m/m$ and $\hat{c}_m$ are integers are divisible by
$n$
for the case of
\be
x_1^n + \cdots + x_n^n = n \psi x_1\cdots x_n.
\ee

\subsection{Discussions}

In this paper we have considered the variation of Mahler measures of some polynomials
and define a function $Q$.
We have identified $Q$ with the local mirror map of a related Picard-Fuchs system,
which corresponds to some local Calabi-Yau geometry.
Some conjectures are made about some integrality properties
of the expression of $Q$ in terms of $q$ and the expression of $q$ in terms of $Q$.
Their enumerative meaning is not clear at present.

In \cite{Sti}
Beauville's semistable families of elliptic curves over $\bP^1$ with four singular fibers were considered.
It is interesting to extend the discussion in this paper to semistable families of Calabi-Yau $n$-folds over $\bP^1$
for $n > 1$.
In this paper we have only considered hypergeometric series in one variable.
Another direction for extension is to consider multivariate hypergeometric series.
We hope to address these problems in subsequent research.

\vspace{.1in}
{\em Acknowledgements}.
This research is supported in part by NSFC grants (10425101 and 10631050)
and a 973 project grant NKBRPC (2006cB805905).


\begin{thebibliography}{99}


\bibitem{CKYZ}
T.-M. Chiang, A. Klemm, S.-T. Yau, E. Zaslow,
{\em Local Mirror Symmetry: Calculations and Interpretations},
Adv.Theor.Math.Phys. {\bf 3} (1999), 495-565.


\bibitem{Del}
E. Delaygue,
{\em Crit\'ere pour l'int\'egralit\'e des coefficients de
Taylor des applications miroir},
arXiv:0912.3776.

\bibitem{Good}
I. J. Good, {\em Generalizations to several variables of Lagrange's expansion, with
applications to stochastic processes},  Proc. Cambridge Philos. Soc. {\bf 56}  (1960),
367-380.

\bibitem{GV}
R. Gopakumar, C. Vafa, {\em M-theory and topological strings-II},
hep-th/9812127.

\bibitem{Kon}
Y. Konishi, {\em Integrality of Gopakumar-Vafa invariants of toric Calabi-Yau threefolds},
Publ. Res. Inst. Math. Sci. {\bf 42} (2006), no. 2, 605-648, arXiv:math/0504188.

\bibitem{KSV}
M. Kontsevich, A. Schwarz, V. Vologodsky,
{\em Integrality of instanton numbers and $p$-adic B-model},
Phys. Lett. B {\em 637} (2006), no. 1-2, 97-101, hep-th/0603106.

\bibitem{KR1}
C. Krattenthaler,T. Rivoal,
{\em Multivariate $p$-adic formal congruences and
integrality of Taylor coefficients of mirror maps},
arXiv:0804.3049.

\bibitem{KR2}
C. Krattenthaler, T. Rivoal,
{\em On the integrality of the Taylor coefficients of mirror maps},
Duke Math. J. {\bf 151} (2010), 175-218,
arXiv:0907.2577.

\bibitem{KR3}
C. Krattenthaler, T. Rivoal,
{\em On the integrality of the Taylor coefficients of mirror maps, II},
Commun. Number Theory Phys. {\bf 3} (2009), 555-591,
arXiv:0907.2578.


\bibitem{Lia-Yau1}
B. H. Lian, S.-T. Yau, {\em Mirror maps, modular relations and hypergeometric series I}, appeared
as {\em Integrality of certain exponential series}
, in: Lectures in Algebra and Geometry, Proceedings of
the International Conference on Algebra and Geometry, Taipei, 1995, M.-C. Kang (ed.), Int. Press,
Cambridge, MA, 1998, pp. 215-227.

\bibitem{Lia-Yau2}
B. H. Lian, S.-T. Yau,
{\em The nth root of the mirror map}, in: Calabi-Yau varieties and mirror symmetry,
Proceedings of the Workshop on Arithmetic, Geometry and Physics around Calabi-Yau Varieties
and Mirror Symmetry, Toronto, ON, 2001, N. Yui and J. D. Lewis (eds.), Fields Inst. Commun., 38,
Amer. Math. Soc., Providence, RI, 2003, pp. 195-199.

\bibitem{OV}
H. Ooguri, C. Vafa,
{\em Knot invariants and topological strings},
Nuclear Phys. B {\bf 577} (2000),
419-438.

\bibitem{Pen}
P. Peng, {\em A simple proof of Gopakumar-Vafa conjecture for local toric Calabi-Yau manifolds},
Comm. Math. Phys. {\bf 276} (2007), no. 2, 551-569,  arXiv:math/0410540.

\bibitem{RV}
F. Rodriguez Villegas,  \em{Modular Mahler measures I},
 Topics in number theory (University Park, PA, 1997), S. Ahlgren,  G. Andrews, K. Ono (eds) 17--48,
Math. Appl., 467,
Kluwer Acad. Publ., Dordrecht, 1999.

\bibitem{Sti}
J. Stienstra,
{\em Mahler measure variations, Eisenstein series and instanton expansions},
in Mirror symmetry. V, 139-150, AMS/IP Stud. Adv. Math., 38, Amer. Math. Soc., Providence, RI, 2006.
 arXiv:math/0502193.
\bibitem{Zag}
D. Zagier, \em{Integral solutions of Ap\'ery-like recurrence equations},
in Groups and symmetries, 349-366,
CRM Proc. Lecture Notes, 47, Amer. Math. Soc., Providence, RI, 2009.

\bibitem{Zud}
W. Zudilin, {\em Integrality of power expansions related to hypergeometric series}, Mathematical Notes
{\bf 71.5} (2002), 604-616.

\bibitem{Zho}
J. Zhou,
{\em Some integrality properties in local mirror symmetry},
arXiv:1005.3243.

\end{thebibliography}
\end{document}